\documentclass{amsart}

\usepackage{amssymb}
\usepackage{graphicx}
\usepackage{MnSymbol}

\usepackage{enumerate}

\usepackage{amsmath,amscd}

\usepackage{amsthm}

\title[A permutation group acting transitively]{A permutation group acting transitively on certain collections of models}

\author{Samuel M. Corson}
\address{E. T. S. I. I. Universidad Polit\'{e}cnica de Madrid, Jos\'{e} Guti\'{e}rrez Abascal 2, 28006 Madrid, Spain}
\email{sammyc973@gmail.com}
\author{Saharon Shelah}
\address{Einstein Institute of Mathematics, The Hebrew University of Jerusalem, Jerusalem 91904 Israel}
\address{Department of Mathematics, Rutgers University, Piscataway, NJ 08854 USA}
\email{shelah@math.huji.ac.il}

\theoremstyle{definition}\newtheorem{theorem}{Theorem}
\theoremstyle{definition}
\theoremstyle{definition}

\theoremstyle{definition}

\theoremstyle{definition}

\theoremstyle{definition}
\theoremstyle{definition}

\theoremstyle{definition}

\theoremstyle{definition}\newtheorem{corollary}[theorem]{Corollary}
\theoremstyle{definition}
\theoremstyle{definition}\newtheorem{definition}{Definition}
\theoremstyle{definition}
\theoremstyle{definition}
\theoremstyle{definition}
\theoremstyle{definition}
\theoremstyle{definition}\newtheorem{lemma}[theorem]{Lemma}
\theoremstyle{definition}
\theoremstyle{definition}
\theoremstyle{definition}
\theoremstyle{definition}
\theoremstyle{definition}\newtheorem{definitions}[theorem]{Definitions}
\theoremstyle{definition}\newtheorem{notation}[theorem]{Notation}

\newcommand{\Sym}{\operatorname{Sym}}
\newcommand{\dom}{\operatorname{dom}}
\newcommand{\Aut}{\operatorname{Aut}}
\newcommand{\It}{\operatorname{It}}
\newcommand{\ran}{\operatorname{ran}}
\newcommand{\set}{\operatorname{set}}
\newcommand{\fix}{\operatorname{fix}}

\begin{document}



\thanks{The work of the first author was supported by Basque Government Grant IT1483-22 and Spanish Government Grants PID2019-107444GA-I00 and PID2020-117281GB-I00.  The work of the second author is supported by ISF 2320/23: The Israel Science Foundation (ISF) (2023-2027).  Paper number 1263 on Shelah's archive.}

\begin{abstract}
It is shown, from $\sigma$-centered Martin's Axiom, that there exists a proper dense subgroup of the symmetric group on a countably infinite set whose natural action on sufficiently flexible relational structures is transitive.  This allows us to give consistent positive answers to some questions of Peter M. Neumann from the 1980s.
\end{abstract}

\maketitle

\begin{section}{Introduction}
Let $\Omega$ be a countably infinite set and $\Sym(\Omega)$ be the group of all bijections on $\Omega$.  Recall that $\Sym(\Omega)$ has the natural topology of pointwise convergence, under which it is a Polish group.  In this paper we will consider relational models (in the sense of universal algebra) $\mathbb{M} = (\Omega, \{R_i\}_{i \in I})$ whose universe is $\Omega$.  Each permutation $f \in \Sym(\Omega)$ trivially gives an isomorphism $$f: \mathbb{M} = (\Omega, \{R_i\}_{i \in I}) \rightarrow \mathbb{M}' = (\Omega, \{R_i'\}_{i \in I})$$ where for each $i \in I$, $R_i'(f(a_0), \ldots, f(a_{n-1}))$ if and only if $R_i(a_0, \ldots, a_{n-1})$.  Thus the action of $\Sym(\Omega)$ on $\Omega$ extends to a transitive action on the collection $\overline{\mathbb{M}}$ of those elements of the isomorphism class of a model $\mathbb{M}$ whose universe is $\Omega$ (recall that a group action $G \curvearrowright X$ is \emph{transitive} if for each $x, x' \in X$ there exists $g \in G$ such that $gx = x'$).  It is natural to ask whether the restriction of this action to some proper (dense) subgroup $G < \Sym(\Omega)$ is also transitive.

To manage the reader's expectations, we consider a simple situation in which this can fail.  Take $\mathbb{M} = (\Omega, \{<\})$ where $<$ is a linear order on $\Omega$ which is order isomorphic to the order on the set $\omega$ of natural numbers.  It is straightforward to see that if $G \leq \Sym(\Omega)$ has transitive action on $\overline{\mathbb{M}}$ then $G = \Sym(\Omega)$.  So, to allow a proper subgroup to have a transitive action on $\overline{\mathbb{M}}$ there should be some reasonable flexibility in the automorphisms of $\mathbb{M}$.  The following definition conveys such a notion.

\begin{definition}
We will say that a model $\mathbb{M} = (\Omega, \{R_i\}_{i \in I})$ is \emph{inductively flexible} if there exists a collection $\mathcal{F}$ of partial automorphisms of $\mathbb{M}$ (i.e. each element of $\mathcal{F}$ is a restriction of an automorphism on $\mathbb{M}$) such that

\begin{enumerate}

\item each $f \in \mathcal{F}$ has finite domain;

\item $\emptyset \in \mathcal{F}$;

\item $f \in \mathcal{F}$ implies $f^{-1} \in \mathcal{F}$; and 

\item for each $f\in \mathcal{F}$ and $a \in \Omega \setminus \dom(f)$, the set $\{b \in \Omega : f \cup \{(a, b)\} \in \mathcal{F}\}$ is infinite.

\end{enumerate}
\end{definition}

Under a mild set-theoretic assumption we can find a \emph{single} proper dense  subgroup of $\Sym(\Omega)$ whose action on $\overline{\mathbb{M}}$ is transitive for \emph{every} such $\mathbb{M}$.

\begin{theorem}\label{main} (ZFC + MA($\sigma$-centered))  There exists a proper, dense subgroup $G < \Sym(\Omega)$ such that if $\mathbb{M}$ is inductively flexible then $G$ acts transitively on $\overline{\mathbb{M}}$.
\end{theorem}

The assertion MA($\sigma$-centered) is equivalent to the equality $\mathfrak{p} = 2^{\aleph_0}$ \cite{Bell}, and is a consequence of Martin's Axiom, which is in turn a consequence of the continuum hypothesis.  Thus Theorem \ref{main} gives the consistency, with ZFC, of the existence of such a subgroup $G$.  One can think of many relational structures which are inductively flexible, especially among Fra\"{i}ss\'{e} limits.  Two examples of inductively flexible models are an order on $\Omega$ isomorphic to $\mathbb{Q}$ and a partition of $\Omega$ into infinitely many infinite sets.  Theorem \ref{main} therefore allows us to derive consistent positive solutions to some old Kourovka Notebook questions \cite[9.41 c) and 9.42]{KhMaz} of Peter M. Neumann (see Corollaries \ref{Qtype} and \ref{sections}).

We give some technical preliminary arguments in Section \ref{stind} and then derive the main results in Section \ref{mainconsequences}.

\end{section}

\begin{section}{Strongly independent sets} \label{stind}

In this section we consider elements in the free group $F = F(\{x_n\}_{n < \omega})$ of countably infinite rank.  We use the convention that an element $W \in F$ is a \emph{reduced} word, meaning that if $x_i$ appears in $W$ then $x_i^{-1}$ does not appear immediately before or after it in $W$.  Thus an element of $F$ is \emph{nontrivial} if it does not represent the trivial element, i.e. if it is a word of length at least one.  We write $W(x_0, \ldots, x_n)$ to mean that the only letters which appear in $W$ are in the set $\{x_0^{\pm 1}, \ldots, x_n^{\pm 1}\}$, but possibly only a proper subset appears.  We will say that a reduced word $W(x_0, \ldots, x_n)$ \emph{uses $x_i$} if either $x_i$ or $x_i^{-1}$ appears in $W$.  Note that if $W(x_0, x_1, \ldots, x_n)$ uses $x_0$ then $W$ can be decomposed uniquely as an alternating concatenation of maximal nontrivial subwords $U(x_1, \ldots, x_n)$ which do not use $x_0$ and maximal nontrivial words $x_0^k$ with $k \in \mathbb{Z} \setminus \{0\}$: $$W \equiv (U_{m + 1}(x_1, \ldots, x_n))x_0^{k_m}\cdots x_0^{k_1}U_1(x_1, \ldots, x_n)x_0^{k_0}(U_0(x_1, \ldots, x_n))\eqno{(\dagger)}$$ where the rightmost term $U_0$ and/or the leftmost term $U_{m + 1}$ might not exist.  For such a decomposition we write $L'(W) = (\sum_{i = 0}^m |k_m|) + J(W)$ where $J(W) = m + 2$ in case both $U_0$ and $U_{m+1}$ exist, $J(W) = m + 1$ in case exactly one of $U_0$ and $U_{m + 1}$ exist, and $J(W) = m$ in case neither $U_{m + 1}$ nor $U_0$ exist.  Thus, $L'(W)$ adds the number of occurences of $x_0$ and of $x_0^{-1}$ and of subwords of form $U_i$ in the decomposition $(\dagger)$.  Thus, for a word $W(x_0, \ldots, x_n)$ which uses $x_0$ we can write $$W \equiv V_{L'(W) - 1}V_{L'(W) - 2} \cdots V_0\eqno{(\dagger\dagger)}$$ where each $V_i$ is $x_0$, or $x_0^{-1}$, or one of the $U_i$ in the decomposition $(\dagger)$.  

For the next definition, we take $c$ to be a symbol such that $c \notin \Omega$.  Also, a \emph{partial bijection on $\Omega$} is a function $f$ such that for some $X \subseteq \Omega$ and $g \in \Sym(\Omega)$ we have $f = g \upharpoonright X$ (we allow $X =\Omega$).

\begin{definition}  Suppose that $W(x_0, \ldots, x_n)$ is a word which uses $x_0$ and $f_0, f_1, \ldots, f_n$ is a list of partial bijections on $\Omega$.  We say that a tuple $(t_{L'(W)}, \ldots, t_1, t_0)$ is an \emph{itinerary of $W$ with respect to $f_0, \ldots, f_n$} provided

\begin{enumerate}

\item for each $0 \leq i \leq L'(W)$ we have $t_i \in \Omega \sqcup \{c\}$, and $t_i \in \Omega$ for some $i$;

\item if $t_i \in \Omega$ and $0 \leq i < L'(W)$ then

\[
t_{i + 1} = \left\{
\begin{array}{ll}
 V_i(f_0, \ldots, f_n)) t_i
                                            & \text{if } t_i\in \dom(V_i(f_0, \ldots, f_n)), \\
c 
                                            & \text{otherwise };
\end{array}
\right.
\]

\item if $t_{i + 1} \in \Omega$ and $0 \leq i < L'(W)$ then

\[
t_i = \left\{
\begin{array}{ll}
 (V_i(f_0, \ldots, f_n)))^{-1} t_{i + 1}
                                            & \text{if } t_{i + 1} \in \ran(V_i(f_0, \ldots, f_n)), \\
c 
                                            & \text{otherwise }.
\end{array}
\right.
\]

\item if $t_j = c$ and $t_i \in \Omega$ for some $i < j$ then $t_{\ell} = c$ for all $L'(W) \geq \ell \geq j$;

\item if $t_j = c$ and $t_i \in \Omega$ for some $j < i$ then $t_{\ell} = c$ for all $0 \leq \ell \leq j$.

\end{enumerate}

\end{definition}

The idea of an itinerary is to take an element of $\Omega$ and place it somewhere in the word and move it through the word backwards and forwards as much as possible, with a symbol $c$ showing where the movement will not be defined.  The following is clear.

\begin{lemma}\label{itineraryisdeterminedbycoord}  Suppose $W(x_0, \ldots, x_n)$ uses $x_0$, the functions $f_0, \ldots, f_n$ are partial bijections on $\Omega$ and $(t_{L'(W)}, \ldots, t_0)$, and $(t_{L'(W)}', \ldots, t_0')$ are two itineraries of $W$ with respect to $f_0, \ldots, f_n$ with $t_i = t_i' \in \Omega$ for some $0 \leq i \leq L'(W)$.  Then the two itineraries are equal.  
\end{lemma}

\begin{notation}\label{itset}  Suppose that $W(x_0, \ldots, x_n)$ uses $x_0$ and the functions $f_0, \ldots, f_n$ are partial bijections on $\Omega$.  We use $\It(W, f_0, \ldots, f_n)$ to denote the set of all itineraries of $W(x_0, \ldots, x_n)$ with respect to $f_0, \ldots, f_n$.  For $L'(W) \geq i \geq 0$ and $a \in \Omega$ we let $\It(W, f_0, \ldots f_n, i, a)$ be the unique (by Lemma \ref{itineraryisdeterminedbycoord}) element of $\It(W, f_0, \ldots, f_n)$ which has $t_i = a$.  If $\overline{t} = (t_{L'(W)}, \ldots, t_0) \in \It(W, f_0, \ldots, f_n)$ then let $\set(\overline{t})$ be the set $\{t_{L'(W)}, \ldots, t_0\} \setminus \{c\}$.

\end{notation}

\begin{definition}
If $\mathbb{M} = (\Omega, \{R_i\}_{i \in I})$ is inductively flexible, and $\mathcal{F}$ witnesses this, then we say that $(\mathbb{M}, \mathcal{F})$ is a \emph{flexible pair}.
\end{definition}

We now give a technical lemma.

\begin{lemma}\label{essentiallem}  Suppose that 

\begin{enumerate}
\item $(\mathbb{M}, \mathcal{F})$ is a flexible pair;

\item $h, f_1, \ldots, f_n \in \Sym(\Omega)$;

\item words $W_0(x_0, \ldots, x_n), \ldots, W_s(x_0, \ldots, x_n)$ each utilize $x_0$; and

\item $g \in \mathcal{F}$.

\end{enumerate}

Then each of the following hold.

\begin{enumerate}[(i)]

\item If $a \in \Omega \setminus \dom(g)$ then there exists $b \in \Omega$ such that $g' := g \cup \{(a, b)\} \in \mathcal{F}$, and for any $0 \leq p \leq s$ and $\overline{t'} \in \It(W_p, h \circ g', f_1, \ldots, f_n)$, if $t_i' = t_j' \in \Omega$ for some $0 \leq i < j \leq L'(W_p)$ then there exist $0 \leq i_0 < j_0 \leq L'(W)$ with $\overline{t} = \It(W_p, h \circ g, f_1, \ldots, f_n, i_0, t_{i_0}')$ satisfying $t_{i_0} = t_{j_0} \in \Omega$.

\item If $b \in \Omega \setminus \ran(g)$ then there exists $a \in \Omega$ such that $g' := g \cup \{(a, b)\} \in \mathcal{F}$, and for any $0 \leq p \leq s$ and $\overline{t'} \in \It(W_p, h \circ g', f_1, \ldots, f_n)$, if $t_i' = t_j' \in \Omega$ for some $0 \leq i < j \leq L'(W_p)$ then there exist $0 \leq i_0 < j_0 \leq L'(W)$ with $\overline{t} = \It(W_p, h \circ g, f_1, \ldots, f_n, i_0, t_{i_0}')$ satisfying $t_{i_0} = t_{j_0} \in \Omega$.

\end{enumerate}

\end{lemma}

\begin{proof}  Assume the hypotheses (1) - (4).  For $0 \leq p \leq s$ we let $$W_p \equiv V_{p, L'(W_p) - 1} V_{p, L'(W_p) - 2} \cdots V_{p, 0}$$ be the $(\dagger\dagger)$ decomposition.  For each $0 \leq p \leq s$ we let $X_p^+ = \{0 \leq i \leq L'(W) - 1: V_i \equiv x_0\}$ and $X_p^- = \{0 \leq i \leq L'(W) - 1: V_i \equiv x_0^{-1}\}$.  We shall first prove conclusion (i).  Let $a \in \Omega \setminus \dom(g)$ be given.  For each $i \in X_p^+$ we define $Z_{p, i}$ to be the set of elements $b \in \Omega$ such that $$\set(\It(W_p, h \circ g, f_1, \ldots, f_n, i + 1, h(b))) \cap \set(\It(W_p, h \circ g, f_1, \ldots, f_n, i, a)) = \emptyset.$$  It is straightforward to see that $\Omega \setminus Z_{p, i}$ has cardinality at most $(i + 1)(L'(W_p) + 1)$.  For $i \in X_p^-$ define $Z_{p, i}$ to be the set of those $b \in \Omega$ such that $$\set(\It(W_p, h \circ g, f_1, \ldots, f_n, i + 1, a)) \cap \set(\It(W_p, h \circ g, f_1, \ldots, f_n, i, h(b))) = \emptyset.$$  By similar reasoning we can see that $\Omega \setminus Z_{p, i}$ is finite.

Defining $Z \subseteq \Omega$ by $Z = (\bigcap_{0 \leq p \leq s}\bigcap_{ i \in X_p^+ \cup X_p^-} Z_{p, i}) \setminus \{h^{-1}(a)\}$ we have that $\Omega \setminus Z$ is finite.  So, we may select $b \in Z$ such that $g' = g \cup \{(a, b)\} \in \mathcal{F}$.  We check the conclusion of (i).  Let $0 \leq p \leq s$ and suppose that $\overline{t'} \in \It(W_p, h \circ g', f_1, \ldots, f_n)$ is such that $t_i' = t_j' \in \Omega$ for some $0 \leq i < j \leq L'(W_p)$.  Let $\overline{t''} = \It(W_p, h \circ g, f_1, \ldots, f_n, i, t_i')$.  If $t_j' = t_j''$, then we may conclude by letting $j_0 = j$ and $i_0 = i$.  Otherwise, there exists some $j > k \geq i$ such that 

\begin{itemize}

\item $t_i' = t_i'', \ldots, t_k' = t_k''$;

\item $V_{p, k} \equiv x_0$ or $V_{p, k} \equiv x_0^{-1}$; and

\item $V_{p, k}(h \circ g)t_k''$ is undefined.

\end{itemize}

\noindent So either $t_k' = a$ and $t_{k + 1}' = h(b)$ (in case $V_{p, k} \equiv x_0$), or $t_k' = h(b)$ and $t_{k + 1}' = a$ (in case $V_{p, k} \equiv x_0^{-1}$).  Let $\overline{t^{(3)}} = \It(W_p, h \circ g, f_1, \ldots, f_n, k + 1, t_{k + 1}')$.  If $t_j^{(3)} = t_j'$ then $$t_j' \in \set(\It(W_p, h \circ g, f_1, \ldots, f_n, i + 1, h(b))) \cap \set(\It(W_p, h \circ g, f_1, \ldots, f_n, i, a))$$ in case $V_{p, k} \equiv x_0$ and $$t_j' \in \set(\It(W_p, h \circ g, f_1, \ldots, f_n, i + 1, a)) \cap \set(\It(W_p, h \circ g, f_1, \ldots, f_n, i, h(b)))$$ in case $V_{p, k} \equiv x_0^{-1}$.  Either way we have contradicted our choice of $b \in Z$.  Thus it is the case that $t_j^{(3)} = c$, and taking $m \geq k + 1$ maximal such that $t_{m}^{(3)} \neq c$ we have

\begin{itemize}

\item $t_{k + 1}' = t_{k + 1}^{(3)}, \ldots, t_m' = t_m^{(3)}$;

\item $V_{p, m} \equiv x_0$ or $V_{p, m} \equiv x_0^{-1}$; and

\item $V_{p, m}(h \circ g)t_m^{(3)}$ is undefined.

\end{itemize}

In particular $j > k + 1$ since $t_j^{(3)} = c$.  We claim that $m \neq k + 1$.  For contradiction we assume $m = k + 1$, so that $V_{p, k + 1} \equiv V_{p, k}$ (since the word $W_p$ is reduced).  If $V_{p, k} \equiv x_0$ then since $V_{p, k + 1}(h \circ g)t_{k + 1}^{(3)}$ is not defined and $V_{p, k + 1}(h \circ g')t_{k + 1}^{(3)}$ is defined, we see that $t_{k + 1}^{(3)} = a$.  By the same reasoning we have $t_{k}'' = a$, and so $t_{k + 1}^{(3)} = h(b)$ and $h(b) = a$, contradicting our choice of $b$.  On the other hand, if $V_{p, k} \equiv x_0^{-1}$ then $t_{k + 1}^{(3)} = h(b)$ (since $V_{p, m}(h \circ g)$ is undefined at $t_{k + 1}^{(3)}$ but $V_{p, m}(h \circ g')$ is defined at $t_{k + 1}^{(3)}$).  By the same reasoning $t_k'' = h(b)$ and therefore $t_{k + 1}^{(3)} = a = h(b)$ contradicting the choice of $b \in Z$.  Thus indeed $m > k + 1$.

If $t_m^{(3)} = t_{k + 1}^{(3)}$ then we may conclude the argument by letting $j_0 = m$ and $i_0 = k + 1$.  Assuming the inequality $t_m^{(3)} \neq t_{k + 1}^{(3)}$, we get that $V_{p, m} \equiv V_{p, k}$.  If $V_{p, k} \equiv x_0$ then $k \in X_p^+$ and $t_k'' = a = t_m^{(3)}$ and we get $$a \in \set(\It(W_p, h \circ g, f_1, \ldots, f_n, k + 1, h(b))) \cap \set(\It(W_p, h \circ g, f_1, \ldots, f_n, k, a))$$ which is contrary to having $b \in Z$.  If $V_{p, k} \equiv x_0^{-1}$ then $k \in X_p^-$ and $$h(b) \in \set(\It(W_p, h \circ g, f_1, \ldots, f_n, k + 1, a)) \cap \set(\It(W_p, h \circ g, f_1, \ldots, f_n, k, h(b)))$$ which also contradicts the choice of $b \in Z$.  Thus we have checked that (i) holds.

The proof of conclusion (ii) is entirely analogous and we give the sketch.  We have $X_p^+$ and $X_p^-$ defined as before, for each $0 \leq p \leq s$.  Let $b \in \Omega \setminus \ran(g)$ be given.  For each $i \in X_p^+$ define $Y_{p, i}$ to be the set of elements $a \in \Omega$ such that $$\set(\It(W_p, h \circ g, f_1, \ldots, f_n, i + 1, h(b))) \cap \set(\It(W_p, h \circ g, f_1, \ldots, f_n, i, a)) = \emptyset.$$  For $i \in X_p^-$ we let $Y_{p, i}$ be the set of $a \in \Omega$ such that $$\set(\It(W_p, h \circ g, f_1, \ldots, f_n, i + 1, a)) \cap \set(\It(W_p, h \circ g, f_1, \ldots, f_n, i, h(b))) = \emptyset.$$  Let $Y = (\bigcap_{0 \leq p \leq s}\bigcap_{ i \in X_p^+ \cup X_p^-} Y_{p, i}) \setminus \{h(b)\}$.  It is easy to see that $\Omega \setminus Y$ is finite, so select $a \in Y$ such that $g' := g \cup \{(a, b)\} \in \mathcal{F}$.  Now follow the proof as in (i), with the obvious slight modifications, to check (ii). 

\end{proof}

\begin{definition}\label{stronglyindependent}  We say a subset $\mathcal{G} \subseteq \Sym(\Omega)$ is \emph{strongly independent} if for every nontrivial $W(x_0, \ldots, x_{\ell}) \in F$ and pairwise distinct elements $f_0, \ldots, f_{\ell} \in \mathcal{G}$, the element $W(f_0, \ldots, f_{\ell}) \in \Sym(\Omega)$ has only finitely many fixed points in $\Omega$.  We will use $\fix(f)$ to denote the set of fixed points of $f \in \Sym(\Omega)$.
\end{definition}

\begin{lemma}\label{onlycountablestronglyind}  There exists a countably infinite, strongly independent set $\mathcal{J}$ which is dense in $\Sym(\Omega)$.
\end{lemma}

\begin{proof}
Let $\mathbb{M} = (\Omega, \emptyset)$ be the relational model having no relations.  Then $\Aut(\mathbb{M}) = \Sym(\Omega)$ and letting $\mathcal{F}$ be the set of all partial bijections on $\Omega$, the ordered pair $(\mathbb{M}, \mathcal{F})$ is clearly a flexible pair.  Let $\mathcal{J}_n = \{f_1, \ldots, f_n\}$ be a (possibly empty) finite strongly independent set with $f_1, \ldots, f_n$ pairwise distinct.  We show that if $q \in \mathcal{F}$ then there exists $\overline{q} \in \Sym(\Omega)\setminus \mathcal{J}_n$ such that $\overline{q} \upharpoonright \dom(q) = q$ and $\mathcal{J}_n \cup \{\overline{q}\}$ strongly independent.  This immediately implies the lemma since we can then enumerate $\mathcal{F} = \{q_{\ell}\}_{\ell \in \omega \setminus \{0\}}$ and inductively define strongly independent sets $\mathcal{J}_n = \{\overline{q_n}\}_{1 \leq \ell \leq n}$ so that the $\overline{q_1}, \ldots, \overline{q_n}$ are pairwise distinct and $\overline{q_{\ell}} \upharpoonright \dom(q_{\ell}) = q_{\ell}$, so $\mathcal{J} = \bigcup_{n \in \omega \setminus \{0\}} \mathcal{J}_n$ is as required.

Let $q \in \mathcal{F}$ be given.  Let $\{W_p(x_0, \ldots, x_n)\}_{p < \omega}$ be an enumeration of the words which use $x_0$ and which do not use $x_i$ for $i > n$.  Let $\Omega = \{a_t\}_{t < \omega}$ be an enumeration.  By iterated use of Lemma \ref{essentiallem} define a sequence $u_s$, $s < \omega$, of elements of $\mathcal{F}$ such that

\begin{enumerate}

\item[$(*)_1$] $q = u_0 \subseteq u_1 \subseteq \cdots \subseteq u_r\subseteq \cdots$;

\item[$(*)_2$] $a_s \in \dom(u_{s + 1}) \cap \ran(u_{s + 1})$;

\item[$(*)_3$] if $p < s$ and $\overline{t'} \in \It(W_p, u_s, f_1, \ldots, f_n)$ is an itinerary for which $t_j' = t_i' \in \Omega$ for some $0 \leq i < j \leq L'(W_p)$ then there exist $0 \leq i_0 < j_0 \leq L'(W)$ such that $\overline{t} = \It(W_p, u_p, f_1, \ldots, f_n, i_0, t_{i_0}')$ satisfies $t_{i_0} = t_{j_0} \in \Omega$.

\end{enumerate}

More precisely, set $u_0 = q$.  Assuming $u_s$ is defined, let $u_s' = u_s$ in case $a_s \in \dom(u_s)$.  Otherwise use Lemma \ref{essentiallem} (i), with $h \in \Sym(\Omega)$ being the identity map, to find $b \in \Omega$ with $u_s' = u_s \cup \{(a_s, b)\} \in \mathcal{F}$ such that if 

\begin{enumerate}

\item $0 \leq p \leq s$;

\item $\overline{t'} \in \It(W_p, u_s', f_1, \ldots, f_n)$; and 

\item $0 \leq i < j \leq L'(W)$ are such that $t_i' = t_j' \in \Omega$

\end{enumerate}

\noindent then there are $0 \leq i_0 < j_0 \leq L'(W)$ such that $\overline{t} = \It(W_p, u_s, f_1, \ldots, f_n, i_0, t_{i_0}')$ satisfies $t_{i_0} = t_{j_0} \in \Omega$.  Next, define $u_s'' = u_s'$ in case $a_s \in  \ran(u_s')$.  Otherwise use Lemma \ref{essentiallem} (ii) to find $a \in \Omega$ such that $u_s'' = u_s' \cup \{(a, a_s)\} \in \mathcal{F}$ and satisfying the comparable conditions.  Set $u_{s + 1} = u_s''$.  Now conditions $(*)_1$ and $(*)_2$ are clear, and condition $(*)_3$ holds for a fixed $p$ by induction on $s \geq p$.  

Write $\overline{q} = \bigcup_{s < \omega} u_s$.  By $(*)_1$ and $(*)_2$ we know $\overline{q} \in \Sym(\Omega)$.   Notice that $\overline{q} \neq f_k$ for each $1 \leq k \leq n$, since the word $W \equiv x_0x_k^{-1}$ is such that $W(\overline{q}, f_1, \ldots, f_n) = \overline{q}f_k^{-1}$ has only finitely many fixed points (at most $|\dom(u_p)| + |\fix(f_k)|$ where $W \equiv W_p$).

Note that in checking strong independence of $\mathcal{J}_n \cup \{\overline{q}\}$ it suffices to prove that if $W(x_0, \ldots, x_n)$ uses $x_0$ then $W(\overline{q}, f_1, \ldots, f_n)$ has finitely many fixed points (since $\mathcal{J}_n$ is already strongly independent, and by permuting the variables in a free word so that $\overline{q}$ is substituted for $x_0$).  Let $p < \omega$.  Write $$W_p \equiv V_{L'(W_p) - 1} \cdots V_0$$ for the $(\dagger\dagger)$ decomposition of $W_p$.  By construction, if $\overline{t} \in \It(W_p, u_p, f_1, \ldots, f_n)$ and $0 \leq i_0 < j_0 \leq L'(W_p)$ are such that $t_{i_0} = t_{i_1} \in \Omega$ then there is some $0 \leq k \leq L'(W_p) - 1$ for which

\begin{itemize}

\item $V_k \equiv x_0$ or $V_k \equiv x_0^{-1}$, and $t_k \in \dom(V_k(u_p))$; or

\item $V_k \not\equiv x_0, x_0^{-1}$ and $V_k(f_1, \ldots, f_n)t_k = t_{k + 1} = t_k$.

\end{itemize}

Suppose that for $\overline{t''} \in \It(W_p, \overline{q}, f_1, \ldots, f_n)$ there exist some $0 \leq i < j \leq L'(W)$ for which $t_i'' = t_j''$.  By $(*)_3$ we select $0 \leq i_0 < j_0 \leq L'(W)$ for which $$\overline{t} = \It(W_p, u_p, f_1, \ldots, f_n, i_0, t_{i_0}'')$$ satisfies $t_{i_0} = t_{j_0} \in \Omega$.  Then we can pick a $k$ for which $V_k \equiv x_0$ or $V_k \equiv x_0^{-1}$, and $t_k \in \dom(V_k(u_p))$; or such that $t_{k + 1} = t_k$ and $V_k \not\equiv x_0, x_0^{-1}$.  Therefore the number of points of $\Omega$ which are fixed by $W_p(\overline{q}, f_1, \ldots, f_n)$ is given by the bound 

$$
\begin{array}{ll}
|\fix(W_p(\overline{q}, f_1, \ldots, f_n))| & \leq  |\{\{0 \leq i \leq L'(W) - 1: V_i \equiv x_0 \vee V_i \equiv x_0^{-1}\}|\cdot |\dom(u_p)|\vspace*{2mm}\\
& +  \sum_{\{0 \leq i \leq L'(W)- 1: V_i \not\equiv x_0, x_0^{-1}\}} |\fix V_i(f_1, \ldots, f_n)|\vspace*{2mm}

\end{array}
$$ since the number on the right bounds the number of itineraries which are not injective into $\Omega$.  Therefore the set of fixed points is finite.

So the set $\mathcal{J}_n \cup \{\overline{q}\}$ is strongly independent and of cardinality $n + 1$.

\end{proof}

\end{section}

\begin{section}{Proof of Theorem \ref{main} and applications}\label{mainconsequences}

We recall some set-theoretic conventions.  Every ordinal is the set of ordinals which are strictly below itself, so $4 = \{0, 1, 2, 3\}$ and $\omega + 2 = \{0, 1, 2, \ldots, \omega, \omega + 1\}$.  Each cardinal is an ordinal (which is the least ordinal having that cardinality).  We use $2^{\aleph_0}$ to denote the cardinality of the continuum.

\begin{lemma}\label{immediatetotheorem} Suppose that for any strongly independent set $\mathcal{G} \subseteq \Sym(\Omega)$, with $|\mathcal{G}| < 2^{\aleph_0}$, and isomorphic inductively flexible structures $\mathbb{M}_1 = (\Omega, \{R_i\}_{i \in I})$ and $\mathbb{M}_2 = (\Omega, \{R_i'\}_{i \in I})$ there exists an isomorphism $f: \mathbb{M}_1 \rightarrow \mathbb{M}_2$ such that $\mathcal{G} \cup \{f\}$ is strongly independent.  Then there exists a subgroup $G \leq \Sym(\Omega)$ which is

\begin{enumerate}

\item proper;

\item dense; and

\item for any pair of isomorphic inductively flexible structures $\mathbb{M}$ and $\mathbb{M}'$ having universe $\Omega$ there exists $f \in G$ which is an isomorphism of $\mathbb{M}$ and $\mathbb{M}'$.

\end{enumerate}

\end{lemma}

\begin{proof}
Assume the hypotheses.  Let $\{(\mathbb{M}_{1, \alpha}, \mathbb{M}_{2, \alpha})\}_{\alpha < 2^{\aleph_0}}$ be an enumeration of all pairs of isomorphic inductively flexible structures which both have universe $\Omega$.  By Lemma \ref{onlycountablestronglyind} let $\mathcal{J}$ be a countably infinite strongly independent set which is dense in  $\Sym(\Omega)$.

By the assumptions of the lemma pick $f_0 \in \Sym(\Omega)$ which is an isomorphism from $\mathbb{M}_{1, 0}$ to $\mathbb{M}_{2, 0}$ such that $\mathcal{J} \cup \{f_0\}$ is strongly independent.  Assuming we have chosen $f_{\beta}$ for all $\beta < \alpha < 2^{\aleph_0}$, select $f_{\alpha} \in \Sym(\Omega)$ which is an isomorphism from $\mathbb{M}_{1, \alpha}$ to $\mathbb{M}_{2, \alpha}$ such that $\mathcal{J} \cup \{f_{\beta}\}_{\beta \leq \alpha}$ is strongly independent.  Let $G = \langle \mathcal{J} \cup \{f_{\alpha}\}_{\alpha < 2^{\aleph_0}} \rangle$.  By construction we know that (3) holds.  As $\mathcal{J} \subseteq G$ we know that (2) holds.  We also know that each nontrivial element of $G$ has finitely many fixed points (since $G$ is generated by a strongly independent set).  Then letting $\Omega = \{a_z\}_{z \in \mathbb{Z}}$ be an enumeration without repetition and 

\[
f(a_z) = \left\{
\begin{array}{ll}
a_{z + 2}
                                            & \text{if } z \in 2\cdot\mathbb{Z}, \\
a_z 
                                            & \text{otherwise}
\end{array}
\right.
\]

\noindent we have $f \in \Sym(\Omega) \setminus G$ and so (1) holds.

\end{proof}

We review some notions regarding partial orders (see \cite[Ch.14]{Halbeisen}).

\begin{definitions}  Let $\mathbb{P} = (P, \leq)$ be a partially ordered set.  A set $Q \subseteq P$ is 

\begin{itemize}

\item \emph{open} if $q \in Q$ and $q \leq p$ imply $p \in Q$;

\item \emph{dense} if for each $p \in P$ there exists $q \in Q$ with $p \leq q$;

\item \emph{centered} if every finite subset of $Q$ has an upper bound in $P$ (i.e. if $\{q_0, \ldots, q_k\} \subseteq Q$ is finite then there exists $p \in P$ with $q_j \leq p$ for each $0 \leq j \leq k$);

\item a \emph{filter} if 

\begin{enumerate}[(i)]

\item $Q \neq \emptyset$;

\item for any $q_0, q_1 \in Q$ there is $q \in Q$ with $q_0, q_1 \leq q$;

\item if $p \leq q \in Q$ then $p \in Q$.

\end{enumerate}

\end{itemize}

\noindent We say $\mathbb{P}$ is \emph{$\sigma$-centered} if $P$ is a countable union of centered sets.  The assertion MA($\sigma$-centered) is the following: 

\noindent Whenever $\mathbb{P} = (P, \leq)$ is a $\sigma$-centered partially ordered set, $P \neq \emptyset$, and $\mathcal{D}$ is a collection of open dense subsets of $P$ with $|\mathcal{D}| < 2^{\aleph_0}$ there exists a filter $Q \subseteq P$ with $Q \cap O \neq \emptyset$ for all $O \in \mathcal{D}$.

\end{definitions}

\begin{lemma}\label{inductionstep} (ZFC + MA($\sigma$-centered))  Suppose $\mathcal{G} \subseteq \Sym(\Omega)$ is strongly independent, with $|\mathcal{G}| < 2^{\aleph_0}$, and $\mathbb{M}_1 = (\Omega, \{R_i\}_{i \in I})$ and $\mathbb{M}_2 = (\Omega, \{R_i'\}_{i \in I})$ are isomorphic inductively flexible structures.  Then there exists an isomorphism $f: \mathbb{M}_1 \rightarrow \mathbb{M}_2$ such that $\mathcal{G} \cup \{f\}$ is strongly independent.
\end{lemma}

\begin{proof}
Assume the hypotheses and let $h: \mathbb{M}_1 \rightarrow \mathbb{M}_2$ be an isomorphism.  Let $(\mathbb{M}_1, \mathcal{F})$ be a flexible pair.  For a reduced word $W$ in a free group $F(\{x_n\}_{n < \omega})$ let $L(W)$ denote the standard length of the word $W$ (for example $L(x_1x_3^{-2}x_1^{-1}) = 4$).  We define a partial order $\mathbb{P} = (P, \leq)$ as follows.  An element $p \in P$ has form $p = (g_p, H_p)$ where $g_p \in \mathcal{F}$ and $H_p$ is a finite subset of $\mathcal{G}$.  For the order $\leq$ on $P$, we write $p \leq p'$ if and only if

\begin{enumerate}[(a)]
\item $g_{p} \subseteq g_{p'}$;

\item $H_p \subseteq H_{p'}$; and

\item if $f_1, \ldots, f_n$ lists the elements of $H_p$ in a pairwise distinct fashion, $W(x_0, \ldots, x_n)$ is a reduced word which uses $x_0$, and if $L(W) \leq n$ then 

\begin{enumerate}[(A)]
\item $\overline{t'} \in \It(W_p, h \circ g_{p'}, f_1, \ldots, f_n)$; and

\item $t_i' = t_j' \in \Omega$ for some $0 \leq i < j \leq L'(W)$

\end{enumerate}

\noindent imply that there exist $0 \leq i_0 < j_0 \leq L'(W)$ such that $$\overline{t} = \It(W_p, h \circ g_p, f_1, \ldots, f_n, i_0, t_{i_0}')$$ has $t_{i_0} = t_{j_0} \in \Omega$.

\end{enumerate}

Notice first that $\mathbb{P}$ is $\sigma$-centered.  To see this, let $\mathbb{E} = \{(p, q): p, q \in P \wedge g_p = g_q\}$.  This is clearly an equivalence relation, having precisely $\aleph_0$ equivalence classes.  If $J$ is an equivalence class of $\mathbb{E}$, say $J = \{p: g_p = g\}$, and $\{(g_1, H_1), \ldots, (g_k, H_k)\} \subseteq Q$ is finite then $(g, \bigcup_{i = 1}^{k} H_i)$ is an upper bound for the set.

Notice as well that given $a \in \Omega$, the set $\mathbb{I}_a^1 = \{p \in P: a \in \dom(g_p)\}$ is open and dense.  That $\mathbb{I}_a^1$ is open is obvious, and that it is dense is immediate from Lemma \ref{essentiallem} (i).  That the set $\mathbb{I}_a^2 = \{p \in P: a \in \ran(g_p)\}$ is open and dense is clear using Lemma \ref{essentiallem} (ii).  For a function $f \in \mathcal{G}$ the set $\mathbb{I}_f = \{p \in P: f \in H_p\}$ is open and dense.  That it is open is clear.  To see that it is dense, let $p \in P$ be given and let $p' = (g_p, H_p \cup \{f\})$, so that indeed $p \leq p' \in \mathbb{I}_f$.

The collection  $\mathcal{D} = \{\mathbb{I}_a^1\}_{a \in \Omega} \cup \{\mathbb{I}_a^2\}_{a \in \Omega} \cup \{\mathbb{I}_f\}_{f \in \mathcal{G}}$ has cardinality less than $2^{\aleph_0}$, so by MA($\sigma$-centered) we select a filter $Q$ which intersects each element of $\mathcal{D}$ nontrivially.  Letting $g = \bigcup \{g_p: p \in Q\}$ it is clear that $g$ is an automorphism of $\mathbb{M}_1$, so the function $f = h \circ g$ is an isomorphism of $\mathbb{M}_1$ with $\mathbb{M}_2$.

It remains to check that $\mathcal{G} \cup \{f\}$ is strongly independent.  Let $\Omega = \{a_s\}_{s < \omega}$ be an enumeration.  Let $f_1 \in \mathcal{G}$ be given and consider the word $W(x_0, x_1) = x_0x_1^{-1}$.  Pick $p_0 \in Q$ such that $|H_{p_0}| \geq 2$ and $f_1 \in H_{p_0}$.  Let $p_1, p_2, \ldots$ be a sequence in $Q$ such that

\begin{itemize}
\item $p_0 \leq p_1 \leq p_2 \leq \cdots$; and

\item $a_s \in \dom(g_{p_s}) \cap \ran(g_{p_s})$;

\end{itemize}

\noindent  Then, reasoning as in the proof of Lemma \ref{onlycountablestronglyind}, the number of fixed points of $W(f, f_1)$ is at most $|\dom(g_{p_0})| + |\fix(f_1)|$ and in particular $f \neq f_1$.

Let $W(x_0, \ldots, x_n)$ be a nontrivial word, which without loss of generality uses $x_0$, and $f_1, \ldots, f_n$ be a list of pairwise distinct elements of $\mathcal{G}$.  Pick $p_0 \in Q$ such that $|H_{p_0}| \geq L(W)$ and $\{f_1, \ldots, f_n\} \subseteq H_{p_0}$.  Picking a sequence $p_0, p_1, \ldots$ in $Q$ with

\begin{itemize}
\item $p_0 \leq p_1 \leq p_2 \leq \cdots$; and

\item $a_s \in \dom(g_{p_s}) \cap \ran(g_{p_s})$;

\end{itemize}

\noindent we can argue as before that $\fix(W(f, f_1, \ldots, f_n))$ is finite.

\end{proof}

Theorem 1 now is immediate by combining Lemma \ref{immediatetotheorem} with Lemma \ref{inductionstep}.  We point out some consequences.  It is straightforward to check that if $<$ is a total linear order on $\Omega$ of order type $\mathbb{Q}$, then $\mathbb{M} = (\Omega, \{<\})$ is inductively flexible.

\begin{corollary}\label{Qtype} (ZFC + MA($\sigma$-centered))  There exists a subgroup $G < \Sym(\Omega)$ (e.g. $G$ from Theorem \ref{main}) which is a dense proper subgroup acting transitively on $\overline{\mathbb{M}} = \overline{(\Omega, \{<\})}$.
\end{corollary}

Define an $\aleph_0$-section on $\Omega$ to be a partition of $\Omega$ into $\aleph_0$ many subsets of cardinality $\aleph_0$.  One can consider an $\aleph_0$-section on $\Omega$ as a collection $\{R_{\ell}\}_{\ell < \omega}$ of unary relations satisfying the following three requirements

\begin{enumerate}
\item[$(s)_1$] $(\forall \ell < \omega)|\{a \in \Omega: R_{\ell} a\}| = \aleph_0$;

\item[$(s)_2$] $(\forall \ell_0 < \ell_1 < \omega)(\forall a \in \Omega) \neg R_{\ell_0}a \vee \neg R_{\ell_1}a$; and

\item[$(s)_3$] $(\forall a \in \Omega)(\exists \ell < \omega) R_{\ell}a$

\end{enumerate}

\noindent but with the caveat that two such collections of relations $\{R_{\ell}\}_{\ell < \omega}$, $\{R_{\ell}'\}_{\ell < \omega}$ are considered identical provided there is some $\sigma \in \Sym(\omega)$ with $R_{\ell}a \Longleftrightarrow R_{\sigma(\ell)}' a$.  It is easy to see that a model $\mathbb{M} = (\Omega, \{R_{\ell}\}_{\ell < \omega})$ with the set of relations satisfying $(s)_1$ - $(s)_3$ is inductively flexible.  Thus $G$ from Theorem \ref{main} will act transitively on $\overline{\mathbb{M}}$ and \emph{a fortiori} $G$ will act transitively on $\aleph_0$-sections.  So we may record the following.

\begin{corollary}\label{sections}  (ZFC + MA($\sigma$-centered))  There exists a subgroup $G < \Sym(\Omega)$ (e.g. $G$ from Theorem \ref{main}) which is a dense proper subgroup acting transitively on $\aleph_0$-sections.
\end{corollary}

Peter M. Neumann asked whether there is a proper subgroup $G < \Sym(\Omega)$ which acts transitively on $\Omega$ and transitively on $\mathbb{Q}$-type orderings on $\Omega$ \cite[Problem 9.42]{KhMaz}.  Since our group $G$ is dense, it is $k$-transitive for each $k < \omega$, and therefore transitive, and so a consistent positive answer to his question follows from Corollary \ref{Qtype}.  Similarly he asks for a proper transitive subgroup of $\Sym(\Omega)$ which is transitive on $\aleph_0$-sections \cite[Problem 9.41 c)]{KhMaz}, and Corollary \ref{sections} supplies a consistent positive answer.

\end{section}


\begin{thebibliography}{A}

\bibitem{Bell}  M. G. Bell, \emph{On the combinatorial principle $P(\mathfrak{c})$}, Fund. Math. 114 (1981), 149-157.

\bibitem{Halbeisen}  L. J. Halbeisen, Combinatorial Set Theory,  Springer (2017).

\bibitem{KhMaz}  E. I. Khukhro and V. D. Mazurov, editors, \emph{Unsolved problems in group theory, the Kourovka notebook}, 20th edition (2022). Sobolev Institute of Mathematics, Russian Academy of Sciences, Siberian Branch, Novosibirsk.

\end{thebibliography}
\end{document}